\documentclass{svjour3}
\usepackage{amssymb}
\usepackage{multirow}
\usepackage{rotating}
\usepackage{cite}
\usepackage{lscape}
\usepackage{caption}
\usepackage{amsmath}
\usepackage{graphicx}
\usepackage{color}

\font\transp=msbm10 scaled\magstep 1 
scaled\magstep 0
\def\transpR{\transp\char'122}
\def\nameuse#1{\csname#1\endcsname}

\def\Rmv[#1]{\nameuse{@ifnextchar} [{\Rmat{#1}}{\Rvec{#1}}}
\def\Rvec#1{\mbox{\transpR}^{#1}}
\def\Rmat#1[#2]{\mbox{\transpR}^{{#1} \times {#2}}}
\def\@begintheorem#1#2{\par\bgroup{\bf #1\ #2. }\it\ignorespaces}
\def\@opargbegintheorem#1#2#3{\par\bgroup{\bf #1\ #2\ (#3). }\it\ignorespaces}
\def\@endtheorem{\egroup}
\begin{document}
\title{ A Fast Eigenvalue Approach for Solving the Trust Region Subproblem with an Additional Linear Inequality }
\titlerunning{}
\author{M. Salahi and A.Taati}
\institute{M. Salahi \and A. Taati \at
              Faculty of Mathematical Sciences, University of Guilan, Rasht, Iran \\
              Tel.: +98-133333901\\
              Fax: +98-1333333509\\
              \email{salahim@guilan.ac.ir, akramtaati@yahoo.com}  }         

\maketitle
\hspace{-1cm}\rule{\textwidth}{0.2mm}


\begin{abstract}
In this paper,  we study the extended trust region subproblem (eTRS) in which the trust region intersects the unit ball with a single linear inequality constraint. By reformulating the Lagrangian dual of eTRS as a two-parameter linear eigenvalue problem, we state a necessary and sufficient condition for its strong duality  in terms of an  optimal solution of a linearly constrained bivariate concave maximization problem. This results in  an efficient algorithm for solving   eTRS of large size whenever the strong duality is detected.  Finally, some numerical experiments are given to show the effectiveness of the proposed method.
\end{abstract}

\section{Introduction}
Consider the following extended trust region subproblem (eTRS)
\begin{align}
p^{*}:=\min \quad &x^TAx -2a^Tx \notag \\
&||x||^2\leq \Delta, \tag{eTRS}\\
&b^Tx\leq c, \notag
\end{align}
where $A\in \mathbb{R}^{n\times n}$ is a symmetric matrix  but not  positive definite,  $a, b\in \mathbb{R}^n$ and $ c, \Delta \in \mathbb{R}$ with $\Delta>0$.  We further assume that  eTRS satisfies the Slater condition, i.e., $\exists  \hat{x}$ with $||\hat{x}||^2<\Delta $ and $b^T\hat{x}< c$. Such model problems  appear in many contexts such as constrained optimization,  when the trust region method is applied to solve nonlinear programming problems with linear inequality constraints \cite{1}, robust optimization under matrix norm or polyhedral uncertainty \cite{4,5,6}, nonlinear optimization problems with discrete variables \cite{2,3},  optimal control and system theory \cite{7,8}.

When $(b,c)=(0,0)$, eTRS reduces to  the classical  trust region subproblem (TRS)
 which plays a cardinal role in the trust region methods for unconstrained optimization \cite{1}. Though  TRS is explicitly nonconvex as $A$ is not necessarily positive semidefinite, it enjoys many powerfull features and can be solved efficiently in practise \cite{9,10,11,12}.  In particular, it admits no  gap with its dual problem and enjoys exact  semidefinite  optimization (SDO) relaxation \cite{9,12,13}. It is well known that generally these properties can not be extended to eTRS. Precisely, the SDO-relaxation  of eTRS is not always  tight \cite{14,15}.  However, it has been shown that the optimal value of eTRS can  be computed via solving a mixed SOCO\footnote{Second order cone optimization}/SDO problem and thus is polynomially solvable \cite{3,15}. Moreover, it has been shown that  the strong duality holds for eTRS under some sufficient conditions. For example, in \cite{14}, the authors have studied  the problem of minimizing an indefinite quadratic function subject to two quadratic inequality constraints and  have shown that the following dimension condition
\begin{align*}\label{dc}
\text{dim}\, \, \text{Null}(A-\lambda_{\text{min}}(A)I_n)\geq 2, \tag{DC}
\end{align*}
together with the Slater condition  is sufficient for zero duality gap  and tightness of the SDO-relaxation of eTRS.
Recently, the eTRS  in which the trust region intersects the unit ball with  a fixed number $m$ linear inequalities, $ b_i^Tx\leq c_i, i=1,.., m$, has received much attention  in the literature \cite{16,17,18}. For example, Jeyakumar and Li in \cite{16} have  proved  that the SDO-relaxation is exact without the Slater constraint qualification whenever an extension of dimension condition (DC) as \begin{align*}\label{newdc}
\text{dim}\, \, \text{Null}(A-\lambda_{\text{min}}(A)I_n)\geq s+1,
\end{align*}
where $s=\text{dim} \,\, \text{span}\{b_1,...,b_m\}$,  is fulfilled. They also  have shown that a set of combined first and second order Lagrange multiplier conditions is necessary and sufficient for global optimality and consequently for the strong duality under the dimension condition together with the Slater condition. Very recently, in \cite{17}, Hsia and Sheu attained  the tightness of SDO-relaxation of  eTRS under the following condition
\begin{align*}
\text{Rank}\,\,([A-\lambda_{\text{min}}(A)I_n\quad   b_1\,\, . . . \,\,\, \, b_m])\leq n-1,\tag{NewDC}
\end{align*}
which is more general than the dimension condition by Jeyakumar and Li.  The paper \cite{18}  also studies the same problem and shows that a particular convex relaxation has no gap for arbitrary $m$ linear inequalities as long as the linear constraints are non-intersecting within the ball.  Moreover, recently in \cite{sf} the authors have given an efficient diagonalization based algorithm for solving eTRS under the dimension condition.  To the best of our knowledge, no necessary and sufficient condition has been introduced for strong duality of eTRS except the one studied in Ai and Zhang \cite{20} for a slightly more general problem consisting of  minimizing a quadratic function over two quadratic constraints, one of which being strictly convex. This condition requires one optimal solution of the corresponding SDO-relaxation and hence, it is not verifiable for large-scale instances.

 All the papers mentioned  above show the theoretical tractability of eTRS through an exact SOCO/SDO reformulation or a tight SDO relaxation.  However, solving a large-scale eTRS instance via these convex relaxations is  not practicable.  Most recently, in \cite{17},  the authors  have proposed an inductive algorithm for solving eTRS via handling at most two trust region subproblems which requires  computing local-nonglobal minimizer of involved TRS. Though their algorithm should be more efficient than SOCO/SDO reformulation, it still seems to be expensive for large-scale eTRS instance since involves computing local-nonglobal minimizer of  TRS.

The main contributions of this paper can be divided into the following two parts.
  \begin{itemize}
  \item[(i)] After deriving a new dual problem to eTRS based on a two-parameter linear eigenvalue problem,   we restate the necessary and sufficient condition for strong duality  by Ai and Zhang \cite{20} in terms of an optimal solution of a linearly constrained  bivariate concave maximization problem. This new condition is easily verifiable for sparse eTRS of large size. Moreover, whenever the strong duality does not hold for eTRS, the new dual problem  should be regarded as a new convex relaxation of eTRS.
       \item[(ii)] We propose an efficient algorithm for solving large sparse eTRS  instance whenever strong duality is detected. The algorithm exploits the sparsity of $A$  and the essential cost of it is the matrix-vector multiplication.
  \end{itemize}

  {\bf Notation}\\
Throughout this paper,  for a symmetric matrix $S$,  $S\succ 0  \,
(S\succeq 0)$ denotes $S $ is positive definite ( positive
semidefinite). For a square matrix $ M$, $M^{\dag}$ denotes to the Moore-Penrose generalized inverse
of $M$, $\text{Range}(M)$ and $\text{Null}(M)$ denotes its range and
null spaces, respectively.

\section{ Parametric Eigenvalue Reformulation of Lagrangian Dual }
In this section,  we reformulate  the Lagrangian dual of eTRS as a linearly constrained concave bivariate  maximization  problem which is based on a two-parameter linear  eigenvalue problem. Our result will then be applied in the next section to enable a method for eTRS of large sizes.

Rendl and Wolkowicz \cite{12} have shown that solving the classical  TRS
\begin{align*}
\min \quad & x^TAx-2a^Tx \notag\\
&||x||^2\leq \Delta \tag{TRS},
\end{align*}
is equivalent to the following parametric eigenvalue  problem
\begin{align*}
\max_t \quad  k(t):=&(\Delta+1)\lambda_{\text{min}}(D(t))-t \\
s.t \quad \quad \quad  \quad & \lambda_{\text{min}}(D(t))\leq 0,
\end{align*}
where $$D(t)=\begin{bmatrix} t&-a^T\\-a &A\end{bmatrix}.$$
The Lagrangian dual of eTRS can be equivalently written as
\begin{align*}
d^{*}:= \max_{\lambda \geq 0}\min_{||x||^2\leq \Delta} x^TAx-2a^Tx+\lambda(b^Tx-c). \tag{D-eTRS}
\end{align*}
Using the fact that the inner minimization problem in (D-eTRS) is equivalent to
\begin{align}
\max_t\quad &(\Delta+1)\lambda_{\text{min}}(D(t,\lambda))-t \notag\\
s.t\quad \,\,\, &\lambda_{\text{min}}(D(t,\lambda))\leq 0,\notag
\end{align}
where
$$
D(t,\lambda)=\begin{bmatrix}
t& (-a+\frac{\lambda}{2}b)^T\\
(-a+\frac{\lambda}{2}b) & A
\end{bmatrix},$$
 (D-eTRS) can be reduced to the following problem
\begin{align}\label{eigp}
d^{*}=\max_{ \lambda\geq 0} \quad & k(t,\lambda):= (\Delta+1)\lambda_{\text{min}}(D(t,\lambda))-t -\lambda c\notag\\
s.t\quad  \,&\quad \quad  \quad  \lambda_{\text{min}}(D(t,\lambda))\leq 0. \tag{NewD-eTRS}
\end{align}
By the assumption that $A$ is not positive definite, the interlacing properties of $D(t,\lambda)$ and $ A$ implies that  the constraint $\lambda_{\text{min}}(D(t,\lambda))\leq 0$ in (\ref{eigp}) is redundant.
Since $\lambda_{\text{min}} (D(t,\lambda))$ is a concave function, then $k(t,\lambda)$ is concave and this shows that  the Lagrangian dual of eTRS is equivalent to a linearly constrained concave bivariate maximization problem.  It is known that the function $k(t,\lambda)$ is differentiable at any $(t,\lambda)$  for which $\lambda_{\text{min}}(D(t,\lambda)) $  has multiplicity $r=1$. However,  in the case that $r>1$, $k(t,\lambda)$ is non-differentiable and the subdifferential of $k(t,\lambda)$, $\partial k(t,\lambda)$, is the set \cite{19}:
\begin{align}\label{sub}
\partial k(t,\lambda)=\text{Conv}(\{\begin{bmatrix} (\Delta+1)y_0^2(t,\lambda)-1\\(\Delta+1)y_0(t,\lambda)b^Tz(t,\lambda)-c\end{bmatrix}\}),
\end{align}
where $\begin{bmatrix}y_0(t,\lambda)\\ z(t,\lambda)\end{bmatrix}$ is a normalized eigenvector corresponding to $\lambda_{\text{min}}(D(t,\lambda))$. In the sequel,  we would like to verify  the relationship between  the dual problem (\ref{eigp})  and  the so-called SDO relaxation of eTRS.  Problem (\ref{eigp}) can be equivalently written as  a   linear semidefinite optimization problem  by adding the variable $y$.
\begin{align}\label{e4}
d^{*}= \max \quad & (\Delta+1)y-t-\lambda c\notag\\
& D(t,\lambda)-y I\succeq 0,\\
&\lambda \geq 0,\, y\leq 0.\notag
\end{align}
The dual of (\ref{e4}) is
\begin{align}\label{rela}
 d_r^{*}:=\min\quad &M_0\bullet X,\notag\\
&M_1\bullet X\leq \Delta,\notag\\
& M_2\bullet X\leq c,\\
&M_3 \bullet X=1,\notag\\
&X\succeq 0,\notag
\end{align}
which is the SDO-relaxation of eTRS and
$$M_0=\begin{bmatrix}0 &-a^T\\-a &A\end{bmatrix},\, M_1=\begin{bmatrix}0 & 0_{1\times n}\\0_{n\times 1}& I\end{bmatrix}, \,M_2=\begin{bmatrix}0& \frac{b^T}{2}\\\frac{b}{2}& 0_{n\times n}\end{bmatrix},\, M_3=\begin{bmatrix}1& 0_{1\times n}\\0_{n\times 1}& 0_{n \times n}\end{bmatrix}.$$
Note that the Slater condition for eTRS implies that problem (\ref{rela}) has a strictly feasible solution. Moreover, it is clear  to see that (\ref{e4}) also has an interior feasible point. Therefore,  strong duality holds for (\ref{e4}) and (\ref{rela}), i.e.,  $d^{*}=d_r^{*}$ and both optimal values are attained.  This implies that  the dual problem (\ref{eigp})  should be regarded as a new convex relaxation of eTRS. Furthermore, the dual of (\ref{rela}) is
\begin{align}\label{6}
\max \quad & - \lambda_1\Delta -\lambda_2c +\gamma\notag\\
& \begin{bmatrix}
-\gamma & (-a+\frac{\lambda_2}{2}b)^T\\
(-a+\frac{\lambda_2}{2}b)& A+\lambda_1I
\end{bmatrix}\succeq 0,\\
&\lambda_1\geq 0,\, \lambda_2\geq 0\notag.
\end{align}
which is also the Lagrangian dual of eTRS and obviously equivalent to (\ref{eigp}).
In the next section, we show that a global optimal solution of eTRS can be derived from the optimal solution of problem
(\ref{eigp}).
\section{ Strong Duality and  Global Optimization}
In this section,  we explain how one can obtain a global optimal solution of eTRS after solving problem (\ref{eigp}). From now on, we  consider the problem (\ref{eigp}) without the redundant constraint,  $\lambda_{\text{min}}(D(t,\lambda))\leq 0$. We proceed by  recalling  the following two theorems concerning the properties of the eigenvalues and the eigenvectors of the parametric matrix $D(t)$ which will be used in the rest of the paper.
\begin{theorem}[\cite{12}]\label{t1}
Consider the TRS problem. Suppose that $A=Q\Lambda Q^T$ be the eigenvalue decomposition of $A$. Let $\lambda_{\text{min}}(A):=\Lambda_{1,1}$  has multiplicity $i$  and define
$$t_0:=\lambda_{\text{min}}(A)+\sum_{j\in\{i|(Q^Ta)_i\neq 0\}}\frac{(Q^Ta)_j^2}{\Lambda_{j,j}(A)-\lambda_{\text{min}}(A)}.$$
Then,
in the easy case, for $t\in \mathbb{R}$, $\lambda_{\text{min}}(D(t))<\lambda_{\text{min}}(A)$ and has multiplicity 1. In the hard case, for $t<t_0$, $\lambda_{\text{min}} (D(t))<\lambda_{\text{min}}(A)$ and has multiplicity 1, for $t=t_0$, $\lambda_{\text{min}}(D(t))=\lambda_{\text{min}}(A) $ and has multiplicity $i+1$ and for $t>t_0$, $\lambda_{\text{min}}(D(t))=\lambda_{\text{min}}(A)$ and has multiplicity $i$.
\end{theorem}
\begin{theorem}[\cite{12}]\label{t2}
Consider the TRS problem. Let $y(t)$ be an eigenvector for $\lambda_{\text{min}}(D(t))$ and let $y_0(t)$ be its first component. Then in the easy case, for $t\in \mathbb{R}$, $y_0(t)\neq 0$. In the hard case,  for $t<t_0$, $y_0(t)\neq 0$, for $t>t_0$, $y_0(t)=0$ and for $t=t_0$, there exists a basis for eigenspace of $\lambda_{\text{min}}(D(t_0))$ such that one eigenvector of this basis satisfies  $y_0(t)\neq0$ and the other eigenvectors satisfy $y_0(t)=0$.
\end{theorem}
The following theorem  states a necessary and sufficient condition for strong duality of eTRS. It  was studied in  Ai and Zhang \cite{20} for a slightly more general problem and  is helpful to obtain the main result of this section.
\begin{theorem}[\cite{20}]\label{t3}
Suppose that eTRS satisfies the Slater condition. Then eTRS has no strong duality if and only if there exist multipliers $\lambda, \mu $ such that the following hold:
\begin{itemize}
\item [1.] $\lambda>0$ and $\mu >0$;
\item [2.] $H(\lambda):=A+\lambda I\succeq 0$, and $\text{Rank} (H(\lambda))=n-1$;
\item [3.] The system of linear equations, $H(\lambda)x=a-\frac{\mu}{2}b$, has two solutions $x_1$ and $x_2$ satisfying $x_i^Tx_i=\Delta $, $i=1,2$, and $(b^Tx_1-c)(b^Tx_2-c)<0$.
\end{itemize}
\end{theorem}
Before giving the main theorem, we need the following lemma.
\begin{lemma}\label{le}
Suppose that $(t^{*},\lambda^{*})$ is an optimal solution of problem (\ref{eigp}). Let $\lambda_{min}(A) $  and  $\lambda_{\text{min}}(D(t^{*},\lambda^{*}))$ have multiplicity $i$ and $r$, respectively. Define
$$t_0:=\lambda_{min}(A)+\sum_{j\in\{i|(Q^T(a-\frac{\lambda^{*}}{2}b))_i\neq 0\}}\frac{(Q^T(a-\frac{\lambda^{*}}{2}b))_j^2}{\Lambda_{j,j}(A)-\lambda_{\text{min}}(A)}.$$
Then:
\begin{itemize}
\item[1.] If $\lambda_{\text{min}}(D(t^{*},\lambda^{*})) <\lambda_{\text{min}}(A)$, then $r=1$ otherwise $r=i+1$.
\item[2.]
Let $y(t^{*},\lambda^{*})$ be an eigenvector for $\lambda_{\text{min}}(D(t^{*},\lambda^{*}))$ and let $y_0(t^{*},\lambda^{*})$ be its first component.  If $\lambda_{\text{min}}(D(t^{*},\lambda^{*})) <\lambda_{\text{min}}(A)$, then, $y_0(t^{*},\lambda^{*})\neq 0$  otherwise  there exists a basis for eigenspace of $\lambda_{\text{min}}(D(t^{*},\lambda^{*}))$ such that one eigenvector of this basis satisfies  $y_0(t^{*},\lambda^{*})\neq0$ and the other eigenvectors satisfy $y_0(t^{*},\lambda^{*})=0$.
\end{itemize}
\end{lemma}
\begin{proof}
\begin{itemize}
\item [1.]
Since $(t^{*},\lambda^{*})$ is an optimal solution of problem (\ref{eigp}), we have
$$t^{*}\in \text{argmax}_t\quad k(t,\lambda^{*}):= (\Delta+1)\lambda_{\text{min}}(D(t,\lambda^{*}))-t -\lambda^{*} c. $$
Hence, by Theorem \ref{t1}, we know if $\lambda_{\text{min}}(D(t^{*},\lambda^{*})) <\lambda_{\text{min}}(A)$, then $r=1$. Now suppose that $\lambda_{\text{min}}(D(t^{*},\lambda^{*})) =\lambda_{\text{min}}(A)$. Then  from Theorem \ref{t1}   the following trust region subproblem is a hard case instance:
\begin{align*}
\min \quad & x^TAx-2(a-\frac{\lambda^{*}}{2}b)^Tx \notag\\
&||x||^2\leq \Delta \notag.
\end{align*}
In this case,  Theorem \ref{t2} implies that the  function $k(t,\lambda^{*})$ is differentiable at any $t>t_0$ with $\frac{\partial k(t,\lambda^{*})}{t}=-1$. Therefore,   $t^{*}\leq t_0$. Furthermore,  it also immediately follows from Theorem \ref{t1}  that the case where $\lambda_{\text{min}}(D(t^{*},\lambda^{*}))=\lambda_{\text{min}}(A)$  corresponds to $t^{*}=t_0$ and consequently, $r=i+1$.
\item [2.] The statement immediately follows from Item 1 and Theorem \ref{t2}.
\end{itemize}
 \end{proof}
 The following is the main result of the paper.
\begin{theorem}\label{t4}
 Let $(t^{*},\lambda^{*})$ be  an optimal solution of problem (\ref{eigp}). Then
\begin{itemize}
\item[1.] Let $\lambda_{\text{min}}(D(t^{*},\lambda^{*}))<\lambda_{\text{min}}(A)$ and let $y(t^{*},\lambda^{*})=[y_0(t^{*},\lambda^{*}), z(t^{*},\lambda^{*})]^T$ be a normalized  eigenvector for $\lambda_{\text{min}}(D(t^{*},\lambda^{*}))$. Then the strong duality holds for eTRS and  $x^{*}=\frac{z(t^{*},\lambda^{*})}{y_0(t^{*},\lambda^{*})} $    is a global optimal solution of eTRS.
\item[2.] Let $\lambda_{\text{min}}(D(t^{*},\lambda^{*}))=\lambda_{\text{min}}(A)$ and let $y(t^{*},\lambda^{*})=[y_0(t^{*},\lambda^{*}), z_0(t^{*},\lambda^{*})]^T$ be a normalized eigenvector for $\lambda_{\text{min}}(D(t^{*},\lambda^{*}))$ with $y_0(t^{*},\lambda^{*})\neq 0$.  Then
      \begin{itemize}
      \item[2.1.] If $\lambda_{\text{min}}(A)<0$  has multiplicity one and $\lambda^{*}>0$ and
       the linear system of equations, $ (A-\lambda_{\text{min}}(A) I) x= a -\frac{\lambda^{*}}{2}b $ has two solutions $x_1$ and $ x_2$ satisfying $x_i^Tx_i=\Delta$, $i=1,2$ and $(b^Tx_1-c)(b^Tx_2-c)<0$, then  eTRS has no strong duality.
      \item[2.2.] Otherwise, the strong duality holds for eTRS and there exists an eigenvector  $z$ corresponding to $\lambda_{\text{min}}(A)$  such that  $x^{*}=\frac{z_0(t^{*},\lambda^{*})}{y_0(t^{*},\lambda^{*})}+z $ is a global optimal solution of eTRS.
      \end{itemize}
\end{itemize}
\end{theorem}
\begin{proof}
\begin{itemize}
\item [1.] By the definition of $y(t^{*},\lambda^{*})$ we have
$$ \begin{bmatrix}
t^{*}& (-a+\frac{\lambda^{*}}{2}b)^T\\
(-a+\frac{\lambda^{*}}{2} b)& A
\end{bmatrix} \begin{bmatrix}
y_0(t^{*},\lambda^{*})\\
z(t^{*},\lambda^{*})
\end{bmatrix} =\lambda_{\text{min}}(D(t^{*},\lambda^{*}))\begin{bmatrix}
y_0(t^{*},\lambda^{*})\\
z(t^{*},\lambda^{*})
\end{bmatrix}, $$
and $||y(t^{*},\lambda^{*})||^2=1$. Expanding these equations gives
\begin{align*}
&t^{*}y_0(t^{*},\lambda^{*})+(-a+\frac{\lambda^{*}}{2}b)^T z(t^{*},\lambda^{*})=\lambda_{\text{min}}(D(t^{*},\lambda^{*}))y_0(t^{*},\lambda^{*}),\\
&y_0(t^{*},\lambda^{*})(-a+\frac{\lambda^{*}}{2}b)+Az(t^{*},\lambda^{*})=\lambda_{\text{min}}(D(t^{*},\lambda^{*}))z(t^{*},\lambda^{*}),\\
& y_0(t^{*},\lambda^{*})^2+z(t^{*},\lambda^{*})^Tz(t^{*},\lambda^{*})=1.
\end{align*}
 Since $\lambda_{\text{min}}(D(t^{*},\lambda^{*}))<\lambda_{\text{min}}(A)$ , then by Lemma \ref{le}, we have $y_0(t^{*},\lambda^{*})\neq 0$. Now set $x^{*}=\frac{z(t^{*},\lambda^{*})}{y_0(t^{*},\lambda^{*})}$. Then
 \begin{align} \label{k1}
 &t^{*}+(-a+\frac{\lambda^{*}}{2}b)^T x^{*}=\lambda_{\text{min}}(D(t^{*},\lambda^{*})),\notag \\
&(A-\lambda_{\text{min}}(D(t^{*},\lambda^{*}))I)x^{*}=a-\frac{\lambda^{*}}{2} b,\\
&||x^{*}||^2=\frac{1-y_0(t^{*},\lambda^{*})^2}{y_0(t^{*},\lambda^{*})^2}. \notag
\end{align}
  Furthermore, since $\lambda_{\text{min}}(D(t^{*},\lambda^{*}))< \lambda_{\text{min}}(A)$, then we have
\begin{align}
A-\lambda_{\text{min}}(D(t^{*},\lambda^{*}))I \succ 0.
\end{align}
Moreover, recall  that, by Lemma \ref{le}, $\lambda_{\text{min}}(D(t^{*},\lambda^{*}))$ has multiplicity  $r= 1$. Hence, the function $k(t,\lambda)$ is differentiable at $(t^{*},\lambda^{*})$. Now let us  consider the following cases.\\
{\bf Case 1.} $\lambda^{*}>0$.\\
In this case, since  $(t^{*},\lambda^{*})$ is also an unconstrained maximizer of $k(t,\lambda)$, it is necessary that
\begin{align*}
&\frac{\partial k(t^{*},\lambda^{*})}{\partial t}= (\Delta +1)y_0(t^{*},\lambda^{*})^2-1=0,\\
& \frac{\partial k(t^{*},\lambda^{*})}{\partial \lambda}= (\Delta +1)y_0(t^{*},\lambda^{*})b^T z(t^{*},\lambda^{*})-c=0.
\end{align*}
Thus $y_0(t^{*},\lambda^{*})^2=\frac{1}{\Delta+1}$ and  consequently,
\begin{align}
 ||x^{*}||^2=\Delta, \quad  b^Tx^{*}=c.
 \end{align}
 Now we are ready to show that in this case strong duality holds for eTRS and $x^{*}$ is a global optimal solution of it. To see this,  we have the following chain of inequalities:
 \begin{align*}
 p^{*}\geq d^{*}&= \max_{\lambda_1,\lambda_2\geq 0}\min_x x^TAx-2a^Tx+\lambda_1(||x||^2-\Delta)+\lambda_2(b^Tx-c)\\
 &\geq \min_x x^TAx-2a^Tx+\lambda_1^{*}(||x||^2-\Delta)+\lambda_2^{*}(b^Tx-c)\\
 &= x^{*^T}Ax^{*}-2a^Tx^{*}+\lambda_1^{*}(||x^{*}||^2-\Delta)+\lambda_2^{*}(b^Tx^{*}-c)\\
 &=x^{*^T}Ax^{*}-2a^Tx^{*}\\
 &\geq p^{*}
 \end{align*}
 where $\lambda_1^{*}:=-\lambda_{\text{min}}(D(t^{*},\lambda^{*}))$,  $\lambda_2^{*}:=\lambda^{*}$ and the first equality follows from the definition of the Lagrangian dual problem, the second equality follows from (5) and (6), the third equality follows from (7) and the last inequality follows from the primal feasibility of $x^{*}$. Therefore,  we have $p^{*}=d^{*}$, i.e., the strong duality holds for eTRS. In particular, we have $x^{*^T}Ax^{*}-2a^Tx^{*}=p^{*}$ and so $x^{*}$ solves eTRS.\\
{\bf Case 2.} $\lambda^{*} =0$. \\In this case, it is necessary that
\begin{align}
&\frac{\partial k(t^{*},\lambda^{*})}{\partial t}= (\Delta +1)y_0(t^{*},\lambda^{*})^2-1=0, \notag \\
&\lim_{(t^{*},\lambda)\longrightarrow (t^{*},\lambda^{*})} \frac{\partial k(t^{*},\lambda)}{\partial\lambda}\leq 0.\notag
\end{align}
This implies that $||x^{*}||^2=\Delta$ and $ b^Tx^{*}\leq c$. It remains to show that  strong duality holds for eTRS and $x^{*}$ solves it. The proof of this case is  similar to  Case 1.
\item[2.]
First notice that, by setting  $x^{*}=\frac{z_0(t^{*},\lambda^{*})}{y_0(t^{*},\lambda^{*})}$,  as discussed in Item 1, we have
\begin{align}
&(A-\lambda_{\text{min}}(D(t^{*},\lambda^{*}))I)x^{*}=a-\frac{\lambda^{*}}{2} b,\\
&(A-\lambda_{\text{min}}(D(t^{*},\lambda^{*}))I) \succeq 0. \notag
\end{align}
\begin{itemize}
\item[2.1.]  The statement immediately follows from Theorem \ref{t3}.
\item[2.2.] To prove the strong duality, we use a contradiction argument. Suppose that the strong duality does not hold for eTRS. Then by Theorem \ref{t3}, there exist multipliers $\lambda, \mu $ such that the following hold:
\begin{itemize}
\item [(i)] $\lambda>0$ and $\mu >0$;
\item [(ii)] $H(\lambda):=A+\lambda I\succeq 0$, and $\text{Rank} (H(\lambda))=n-1$;
\item [(iii)] The system of linear equations, $H(\lambda)x =a-\frac{\mu}{2}b$, has two solutions $x_1$ and $x_2$ satisfying $x_i^Tx_i=\Delta $, $i=1,2$, and $(b^Tx_1-c)(b^Tx_2-c)<0$.
\end{itemize}
Specifically, (ii) implies that $\lambda=-\lambda_{\text{min}}(A)$. Furthermore, it follows from (iii) and (8) that $b\in \text{Range}(A-\lambda_{\text{min}}(A)I)$ and hence, it is orthogonal to eigenspace of $\lambda_{\text{min}}(A)$. Moreover,  we know that the solutions $x_1$ and $x_2$ in (iii) necessarily have the form
$x_1= H(\lambda)^\dag(a-\frac{\mu}{2}b)+z_1$ and $x_2=H(\lambda)^\dag (a-\frac{\mu}{2}b)+z_2$ where $z_i$, for $i=1,2$,   is  an eigenvector corresponding to  $\lambda_{\text{min}}(A)$. This implies that  $b^Tx_1-c=b^Tx_2-c=b^TH(\lambda)^\dag(a-\frac{\mu}{2}b)$, a contradiction to the fact that $(b^Tx_1-c)(b^Tx_2-c)<0$. Therefore, we have the strong duality  for eTRS.
On the other hand, equivalency between dual problem (\ref{eigp}) and the Lagrangian dual problem (\ref{6}) implies that  $\lambda_1^{*}:=-\lambda_{\text{min}}(D(t^{*},\lambda^{*}))$ and $\lambda_2^{*}:=\lambda^{*}$ are optimal Lagrange multipliers corresponding to the norm and the linear constraint of eTRS, respectively. Moreover, since strong duality holds for  eTRS and its Lagrangian dual and both of them are solvable, then  there exist some feasible solutions of eTRS, $x^{**}$, such that with $\lambda_1^{*}$ and $\lambda_2^{*}$ satisfy the following optimality conditions:
 \begin{align}
&(A+\lambda_1^{*} I)x^{**}=a-\frac{\lambda_2^{*}}{2} b, \notag\\
& \lambda_1^{*}(||x^{**}||^2-\Delta)=0, \quad \lambda_2^{*}(b^Tx^{**}-c)=0, \\
&(A+\lambda_1^{*} I)\succeq 0. \notag
\end{align}
Considering (8), we find that necessarily
$x^{**}=x^{*}+z$ where $z$ is an eigenvector corresponding to $\lambda_{\text{min}}(A)$. This   completes the proof. Such vector $z$  can be found by the approach discussed in Section 4.\\
\end{itemize}
\end{itemize}
\end{proof}
As a direct consequence of Theorem \ref{t4}, we deduce the following necessary and sufficient condition for the strong duality of eTRS in terms of an optimal solution of (\ref{eigp}).
\begin{corollary}
 Let $(t^{*},\lambda^{*}) $ be an optimal solution of the dual problem (\ref{eigp}).Then eTRS has no strong duality if and only if  the following hold:
\begin{itemize}
\item [1.] $\lambda_{\text{min}}(D(t^{*},\lambda^{*}))=\lambda_{\text{min}}(A)<0$ and $\lambda^{*}>0$;
\item [2.] $\lambda_{\text{min}}(A) $ has multiplicity one  and
       the linear system of equations, $ (A-\lambda_{\text{min}}(A) I) x= a -\frac{\lambda^{*}}{2}b $ has two solutions $x_1$ and $ x_2$ satisfying $x_i^Tx_i=\Delta$, $i=1,2$ and $(b^Tx_1-c)(b^Tx_2-c)<0$.
\end{itemize}
\end{corollary}
\section{Algorithm and Implementation Details}
Section 3  suggests the following algorithm for eTRS.\\
--------------------------------------------------------------------------------------------------------\\
{\bf New Algorithm  } \\
--------------------------------------------------------------------------------------------------------
\begin{enumerate}
\item  Solve  Problem (\ref{eigp}). Let $(t^{*},\lambda^{*})$ be an optimal solution of it.
\item  Compute  $\lambda_{\text{min}}(D(t^{*},\lambda^{*}))$ and corresponding eigenvector
$y(t^{*},\lambda^{*})=[y_0(t^{*},\lambda^{*}), z(t^{*},\lambda^{*})]^T$  with $y_0(t^{*},\lambda^{*})\neq 0$. Set $x^{*}=\frac{z(t^{*},\lambda^{*})}{y_0(t^{*},\lambda^{*})}$.
\item If  $\lambda_{\text{min}}(D(t^{*},\lambda^{*}))<\lambda_{\text{min}}(A)$, then stop and return $x^{*}$  as a global  optimal solution to eTRS; otherwise go to step 4.
 \item If $\lambda_{\text{min}}(A)<0$ has multiplicity one and $\lambda^{*}>0$, define $x_1:=x^{*}+\alpha_1 z$ and $x_2:=x^{*}+\alpha_2 z$ where $\alpha_1$ and $\alpha_2$ are roots of the following quadratic equation
      $$ \alpha^2 +2x^{*^T}z+x^{*^T}x^{*}-\Delta=0,$$
      and $z$ is a normalized eigenvector for $\lambda_{\text{min}}(A)$.
       If $(b^Tx_1-c)(b^Tx_2-c)<0$, then strong duality does not hold for eTRS. So stop and return $k(t^{*},\lambda^{*})$ as a lower bound for eTRS; otherwise go to step 5.

 \item Compute the vector $z$ in Item 2.2 of Theorem \ref{t4} according to the approach described in Subsection 4.1 and return $x^{*}+z$  as a global optimal solution to eTRS.
     \end{enumerate}
 ---------------------------------------------------------------------------------------------------------\\
In what follows,  we discuss implementation details of the  New  Algorithm.
\subsection{How to compute the vector $z$ in item 2.2 of Theorem \ref{t4}}
 For simplicity, we consider the case when $\lambda_{\text{min}}(A)$ has multiplicity, $i\leq 2$. The case where $i>2$ can be handled by the lemma given at the end of this section.

Suppose $y(t^{*},\lambda^{*})=[y_0(t^{*},\lambda^{*}), z_0(t^{*},\lambda^{*})]^T$ and $x^{*}$ are defined as in Item 2 of Theorem \ref{t4} and $\lambda_{\text{min}}(D(t^{*},\lambda^{*}))=\lambda_{\text{min}}(A)$.
To compute the vector $z$ in Item 2.2 of Theorem \ref{t4}, we separately consider the following three cases. The discussion below uses the fact that $||x^{*}||^2\leq \Delta$ (see the Appendix).\\
{\bf Case 1. $\lambda_{\text{min}}(A)$ has multiplicity one and the strong duality holds for eTRS.}
In this case, by Lemma \ref{le}, $\lambda_{\text{min}}(D(t^{*},\lambda^{*}))$ has an orthonormal basis as
\begin{align*}
\left \{ \begin{bmatrix}y_0(t^{*},\lambda^{*})\\z_0(t^{*},\lambda^{*})\end{bmatrix},\begin{bmatrix}0\\z_1(t^{*},\lambda^{*})\end{bmatrix}\right\}.
\end{align*}
  It is not difficult to see that $z_1(t^{*},\lambda^{*})$  forms  an orthonormal basis for  eigenspace of $\lambda_{\text{min}}(A)$ and $x^{*^T}z_1(t^{*},\lambda^{*})=0$. As  strong duality holds,  the optimality conditions (9)  necessarily has a solution of the form $x^{**}=x^{*}+\alpha z_1(t^{*},\lambda^{*})$.  From the equation $||x^{*}+\alpha z_1(t^{*},\lambda^{*})||^2=\Delta$, we obtain $\alpha=\pm\sqrt{\Delta -||x^{*}||^2}$. Then necessarily, at least one of two values of $\alpha$ results in an optimal solution of eTRS.\\
{\bf Case 2. $\lambda_{\text{min}}(A)$ has multiplicity, $i=2$.}\\
In this case, by Lemma \ref{le}, $\lambda_{\text{min}}(D(t^{*},\lambda^{*}))$ has an orthonormal basis as
\begin{align*}
\left \{ \begin{bmatrix}y_0(t^{*},\lambda^{*})\\z_0(t^{*},\lambda^{*})\end{bmatrix},\begin{bmatrix}0\\z_1(t^{*},\lambda^{*})\end{bmatrix},\begin{bmatrix}0\\z_2(t^{*},\lambda^{*})\end{bmatrix}\right\}.
\end{align*}
It is easy to see that,   $z_i(t^{*},\lambda^{*})$, $i=1,2$, form  an orthonormal basis for eigenspace of $\lambda_{\text{min}}(A)$ and we have  $x^{*^T}z_i(t^{*},\lambda^{*})=0$ for $i=1,2$. As  strong duality holds, the optimality conditions (9) has a solution of the form $x^{**}=x^{*}+z$ where $z$ is an eigenvector corresponding to $\lambda_{\text{min}}(A)$. Let us consider the following subcases.\\
{\bf Subcase 1. $\lambda^{*}=0$ and $b^Tx^{*}\leq c$.}\\
First suppose that $b^Tz_i(t^{*},\lambda^{*})\neq 0$, for $i=1,2$.
The following quadratic polynomial equation of variable $\alpha_2$
$$(1+C^2)\alpha_2^2=\Delta-||x^{*}||^2,$$
where $C=\frac{b^Tz_2(t^{*},\lambda^{*})}{b^Tz_1(t^{*},\lambda^{*})}$,
has two real roots. Let $\alpha_2$ be one of them, then $x^{**}=x^{*}-\alpha_2Cz_1(t^{*},\lambda^{*})+\alpha_2z_2(t^{*},\lambda^{*}) $ is a global optimal solution to eTRS since $||x^{**}||^2=\Delta$ and $b^Tx^{**}\leq c$. Next without loss of generality, assume   that $b^Tz_1(t^{*},\lambda^{*})=0$. In this case, obviously, $x^{**}=x^{*}+\alpha z_1(t^{*},\lambda^{*})$ where $\alpha=\pm\sqrt{\Delta -||x^{*}||^2}$, is an optimal solution of eTRS.\\
{\bf Subcase 2.($\lambda^{*}=0 $ and $ b^Tx^{*}>c$) or $\lambda^{*}>0$.}\\
In this case,   the following system of equations with respect to the unknown variables $\alpha_1$ and $\alpha_2$ must have a solution otherwise it is a contradiction with the fact that the optimality conditions is solvable.
\begin{align*}
&\alpha_1^2+\alpha_2^2+||x^{*}||^2=\Delta,\\
&b^Tx^{*}+\alpha_1 b^Tz_1(t^{*},\lambda^{*})+\alpha_2b^Tz_2(t^{*},\lambda^{*})= c,
\end{align*}
which can be easily solved.\\
{\bf Case 3. $\lambda_{\text{min}}(A) $ has multiplicity, $i>2$.}\\
In this case, in view of the following lemma, deflating $A$ by adding a desirable vector $v$, eTRS reduces to the one satisfying Case 2. Then the approach of Case 2 can be applied.

\begin{lemma}
Suppose that $\lambda_{\text{min}}(A)$ has multiplicity, $i>2$. Furthermore, let $u^{*}=(A+\lambda_1^{*}I)^{\dag}(a-\frac{\lambda_2^{*}}{2}b)$  with $||u^{*}||<\Delta$, where $\lambda_1^{*}:=-\lambda_{\text{min}}(A)$ and $\lambda_2^{*}\geq 0$ are the optimal Lagrange multipliers corresponding to the norm and the linear constraint of eTRS, respectively. Moreover, let $v$ be a normalized  eigenvector corresponding to $\lambda_{\text{min}}(A)$ such that $b^Tv=0$.  Then $(x^{*}=u^{*}+z,\lambda_1^{*},\lambda_2^{*})$ where $z\in \text{Null}(A+\lambda_1^{*}I)$ solves eTRS if and only if $(x^{*}=u^{*}+z,\lambda_1^{*},\lambda_2^{*})$ where $z\in \text{Null}(A+\alpha vv^T+\lambda_1^{*}I)$  solves eTRS when $A$ is replaced by $A+\alpha vv^T$ with $\alpha>0$.
\end{lemma}
\begin{proof}
As $\lambda_{\text{min}}(A)$ has multiplicity, $i>2$,  strong duality holds for eTRS.
First suppose that $(x^{*}=u^{*}+z,\lambda_1^{*},\lambda_2^{*})$ where $z\in \text{Null}(A+\lambda_1^{*}I)$  solves eTRS.  This means that $x^{*}$ with $\lambda_1^{*}$ and $\lambda_2^{*}$ satisfy the following optimality conditions:
\begin{align}
&(A+\lambda_1^{*} I)x^{*}=a-\frac{\lambda_2^{*}}{2} b, \\
& \lambda_1^{*}(||x^{*}||^2-\Delta)=0, \quad \lambda_2^{*}(b^Tx^{*}-c)=0,\notag \\
&(A+\lambda_1^{*} I)\succeq 0. \notag
\end{align}
Next, consider the eTRS with $A$ replaced by $A+\alpha vv^T$ where $v$ is a normalized eigenvector corresponding to $\lambda_{\text{min}}(A)$ such that $b^Tv=0$. We note that such vector $v$ always exists since
$$\text{dim}\, \text{Null}([A-\lambda_{\text{min}}(A)I \quad b]^T)=n-\text{Rank}([A-\lambda_{\text{min}}(A) I\quad  b])\geq 1.$$
Since $\lambda_{\text{min}}(A+\alpha vv^T)=\lambda_{\text{min}}(A)$ still has the multiplicity, $i\geq 2$, then the strong duality  also holds for eTRS with $A+\alpha vv^T$. Furthermore, by  the fact that after deflating $A$,  the optimal Lagrange multipliers, $\lambda_1^{*}$ and $\lambda_2^{*}$ are unchanged (we discuss why the deflating $A$ by adding $\alpha vv^T$ to $A$ does not change the optimal Lagrange multipliers in Appendix),  the following optimality conditions  must have a solution.
\begin{align}
&(A+\alpha vv^T+\lambda_1^{*} I)x^{*}=a-\frac{\lambda_2^{*}}{2} b, \notag\\
& \lambda_1^{*}(||x^{*}||^2-\Delta)=0, \quad \lambda_2^{*}(b^Tx^{*}-c)=0, \notag \\
&(A+\lambda_1^{*} I)\succeq 0. \notag
\end{align}
This implies that $x^{*}=\bar{u}^{*}+z$ where $z\in \text{Null}(A+\alpha vv^T+\lambda_1^{*}I)$ and $\bar{u}^{*}=(A+\alpha vv^T+\lambda_1^{*}I)^{\dag}(a-\frac{\lambda_2^{*}}{2}b)$.
Now let $A=PDP^T$ be the  eigenvalue decomposition of $A$ in which $P$ contains $v$ as its first column. Moreover, let $v_k$, $k=1,..,i$  be the columns of $P$  corresponding to  $\lambda_{\text{min}}(A)$. Since $\lambda_1^{*}=-\lambda_{\text{min}}(A)$, it follows from equation (10) that $v_k^T(a-\frac{\lambda_2^{*}}{2}b)=0$, for $k=1,...,i$. Therefore, we have
\begin{align*}
(A+\alpha vv^T+\lambda_1^{*}I)^{\dag}(a-\frac{\lambda_2^{*}}{2}b)&=P(D+\alpha e_1e_1^T+\lambda_1^{*}I)^{\dag}P^T(a-\frac{\lambda_2^{*}}{2}b)\\
&=P(D+\lambda_1^{*}I)^{\dag}P^T(a-\frac{\lambda_2^{*}}{2}b)=u^{*},
\end{align*}
where $e_1$ is the unit vector, showing $x^{*}=u^{*}+z$ where $z\in \text{Null}(A+\alpha vv^T+\lambda_1^{*}I)$ solves eTRS with $A+\alpha vv^T$. Next assume that $x^{*}=u^{*}+z$ with $z\in \text{Null}(A+\alpha vv^T+\lambda_1^{*}I)$ solves eTRS with $A+\alpha vv^T$. We can show that $(x^{*}=u^{*}+z,\lambda_1^{*},\lambda_2^{*})$ where $z\in \text{Null}(A+\lambda_1^{*}I)$ solves eTRS in a similar manner.
\end{proof}

\subsection{ Solving the dual problem (\ref{eigp}) }
To solve the problem ({\ref{eigp}), we use the alternating optimization approach   which maximize  the underlying problem with respect to one variable while keeping the other one fixed. Precisely, instead of solving the original maximization  problem over two variables, the alternating maximization  algorithm solves a sequence of maximization problems over only  one variable. The method consists of the updates
\begin{align}
\lambda^{k}:&=\text{arg max}_{\lambda\geq 0}\,\, k(t^{k-1},\lambda),\\
t^{k}:&=\,\,\text{arg max}_{t\in \mathbb{R}}\,\, k(t,\lambda^{k}),
\end{align}
where the superscript is the iteration counter.
Since the problem (\ref{eigp}) is a convex optimization problem, it is globally convergent \cite{30}. To solve  subproblem (11), we use the 'fminbnd' command in MATLAB. For subproblem (12), noting that  $\max_{t\in \mathbb{R}} k(t,\lambda^{k})$ is equivalent to the following trust region subproblem
\begin{align}\label{t}
\min \quad & x^TAx-2(a-\frac{\lambda^{k}}{2}b)^Tx  \\
& ||x||^2= \Delta,\notag
\end{align}
 we take advantage of the so-called RW algorithm due to Rendl and Wolkowicz for solving the classical trust region subproblem \cite{12}.
 Suppose that $x(t^k,\lambda^k)$ be an optimal solution of (\ref{t}) obtained from RW algorithm.
As the alternating algorithm proceeds, $(t^{k},\lambda^{k})$ converges to $(t^{*},\lambda^{*})$, an optimal solution of problem (\ref{eigp}). Therefore, by Theorem \ref{t4}, if $\lambda_{\text{min}}(D(t^{*},\lambda^{*}))<\lambda_{\text{min}}(A)$,
then $x(t^k,\lambda^k)$ converges to an optimal solution of eTRS as $(t^{k},\lambda^{k})$ converges to $(t^{*},\lambda^{*})$. Moreover, if $\lambda_{\text{min}}(D(t^{*},\lambda^{*}))=\lambda_{\text{min}}(A)$ and the strong duality is detected, then we are in Step 5 of New Algorithm. Then to compute the vector $z$ in Step 5, one can apply the approach discussed in Subsection 4.1 for $ x(t^k,\lambda^{k})$ at termination of alternating method instead of $x^{*}$ defined in New Algorithm  by considering   a slight modification which is noting that $x(t^{k},\lambda^{k})^Tz_i$ is not necessarily zero.

\section{Numerical Experiments}
In this section, we present some numerical experiments to illustrate the effectiveness of the proposed algorithm. To the best of our knowledge, currently there are  no algorithms in the literature specialized for solving large scale eTRS instances. Thus just for some small instances we compare the New Algorithm with the SOCO/SDO reformulation in \cite{3}. Computations are performed in MATLAB 8.1.0.604 on a 1.70 GHz laptop with 4 GB of RAM. To solve the SOCO/SDO problem, we have used SeDuMi 1.3 and  the maximum number of  iterations for alternating method  is set  to be 2. We consider  two classes of test problems as follows.\\ \\
\begin{itemize}
\item  {\bf First class of test problems:}\\
For this class of test  problems, we generate matrix $A$ with multiple smallest eigenvalue as follows. We first generate randomly a sparse  matrix $ A_0\in \mathbb{R}^{(n-m)\times (n-m)}$  via \verb=A=\verb===\verb=sprandsym(n-m,density)=  where $m$ is the multiplicity of $\lambda_{min}(A)$. Next we construct $A=\begin{bmatrix}A_0&O_{(n-m)\times m}\\O_{m\times (n-m)}& (\lambda_{\text{min}}(A_0)-\alpha)I_{m\times m}\end{bmatrix}$ where $\alpha$ is a positive scalar. It is clear that $A\in \mathbb{R}^{n\times n}$ is a symmetric sparse matrix with $\lambda_{\text{min}}(A)=\lambda_{\text{min}}(A_0)-\alpha $ and its multiplicity is $m$. Furthermore, let $J$ be a permutation of $\{1,2,...,n\}$, then $A:=A(J,J)$ which is obtained from $A$ by permuting the rows and columns of $A$, is still a symmetric sparse matrix with $\lambda_{\text{min}}(A)=\lambda_{\text{min}}(A_0)-\alpha $ and its multiplicity is still $m$. Moreover, the vectors $a$ and $b$ are generated randomly via \verb=a=\verb===\verb=10*randn(n,1)= and  \verb=b=\verb===\verb=10*randn(n,1)=, respectively. Finally, we set $\Delta=1$.  Comparison  with the SOCO/SDO relaxation of eTRS on several small instances which we report  difference between the computed optimal objective values of the New Algorithm and SOCO/SDO relaxation, ($|FvalAlg-FvalSOCO/SDO|$) and algorithm run time in second (Time) at termination averaged over the 10 random instances are given in Table 1. As we see, the New Algorithm finds the global optimal solution significantly faster than the SOCO/SDO reformulation. Moreover, for large instances, to justify the efficiency of  the New Algorithm, we report
 KKT1 ($||(A+\lambda_1^{*}I)x^{*}-(a-\frac{\lambda_2^{*}}{2}b))||_{\infty}$), KKT2 ($\lambda_1^{*}(||x^{*}||^2-\Delta)$),  KKT3 ($\lambda_2^{*}(b^Tx^{*}-c)$) in addition to dimension of problem (n), algorithm run  time in second (Time),  at termination averaged over the 10 random instances in Table 2. We note that, since $m=2$, strong duality holds for this set of problems, and so KKT1, KKT2 and KKT3 are the corresponding optimality conditions.  As we see from Table 2, large instances are also  solved in short time and increase in time in comparison to the increase in the dimension, is significantly smaller.   \\

 \item {\bf The second class of test problems:}\\
    To generate the second class of test problems, matrix $A$ and the vector $a$ are constructed  randomly, the vector $b$ is set to  $e_1$, the unit vector,  and $c=1$ \cite{3}. For this class of  test problems,  there is no evidence for  strong duality property in advance, however, the New Algorithm detected strong duality for all test problems. As the previous case, we compare the results of New Algorithm   with the SOCO/SDO relaxation of \cite{3} for small instances in Table 3 and then solve some large instances in Table 4. Similar to the previous class, for small instances, the New Algorithm  is still much faster than SOCO/SDO relaxation and for large instances it achieves the global optimal solution in short time. Moreover, increase in time in comparison to the increase in the dimension, is significantly smaller.
\end{itemize}
     \begin{table}
\caption{Comparison  between the New Algorithm  and SOCO/SDO reformulation for the first class of test problems with $density=1e-2$ and $m=2$.}
\begin{tabular}{l| l| c| c   }
&&$|FvalAlg-FvalSOCO/SDO|$& Time (s) \\
\hline \hline
 n=200& New Algorithm &$0.1274\times 10^{-8}$&0.76\\
\cline{2-2} \cline{4-4}
& SOCO/SDO &&8.31\\
\hline
 n=300& New Algorithm &$3.7120\times 10^{-9}$&0.95\\
\cline{2-2} \cline{4-4}
& SOCO/SDO &&33.38\\
\hline
 n=400& New Algorithm &$7.1402\times 10^{-9}$&1.04\\
\cline{2-2} \cline{4-4}
& SOCO/SDO &&100.21\\
\hline
 n=500& New Algorithm &$6.1021\times 10^{-8}$&1.08\\
\cline{2-2} \cline{4-4}
& SOCO/SDO &&430.11\\

\hline \hline

\end{tabular}

\end{table}

\begin{center}
\begin{table}
\caption{Computational results of the New Algorithm  for the first class of test problems with $density=1e-4$ and $m=2$.}
\begin{tabular}{l| c c c c   }
n&KKT1&KKT2&KKT3&Time(s)\\
\hline \hline
10000& $1.2302\times 10^{-12}$&$2.3041 \times 10^{-12}$&$1.7125\times 10^{-11}$&2.83\\
\hline
20000& $1.3102\times 10^{-8}$&$3.0943 \times 10^{-12}$&$1.5012\times 10^{-8}$&5.12\\
\hline
40000&$7.9012\times 10^{-13}$&$0.9102\times 10^{-11}$&$7.1360\times 10^{-9}$&11.23\\
\hline
60000&$2.0741\times 10^{-11}$&$0.6541\times 10^{-11}$&$6.1420\times 10^{-9}$&23.19\\
\hline
80000&$0.4307\times 10^{-8}$&$2.6841\times 10^{-11}$&$6.6780\times 10^{-8}$&28.60\\
\hline \hline

\end{tabular}
\end{table}
\end{center}

    \begin{table}
\caption{Comparison  between the New Algorithm  and SOCP/SDO relaxation for the second  class of test problems with $density=1e-2$}
\begin{tabular}{l| l| c| c   }
&&$|FvalAlg-FvalSOCO/SDO|$& Time (s) \\
\hline \hline
 n=200& New Algorithm &$1.3045\times 10^{-10}$&0.99\\
\cline{2-2} \cline{4-4}
& SOCO/SDO &&8.31\\
\hline
 n=300& New Algorithm &$4.4578\times 10^{-8}$&1.1\\
\cline{2-2} \cline{4-4}
& SOCO/SDO &&5.98\\
\hline
 n=400& New Algorithm &$9.0124\times 10^{-9}$&1.4\\
\cline{2-2} \cline{4-4}
& SOCO/SDO &&16.81\\
\hline
 n=500& New Algorithm &$6.1021\times 10^{-8}$&1.5\\
\cline{2-2} \cline{4-4}
& SOCO/SDO &&26.15\\

\hline \hline

\end{tabular}

\end{table}

   \begin{center}
\begin{table}
\caption{Computational results of the New Algorithm  for the second class of test problems with $density=1e-4$.}
\begin{tabular}{l| c c c c   }
n&KKT1&KKT2&KKT3&Time(s)\\
\hline \hline
10000& $3.1045\times 10^{-12}$&$0.4521 \times 10^{-11}$&$1.6325\times 10^{-12}$&3.58\\
\hline
20000& $3.1204\times 10^{-10}$&$0.8120 \times 10^{-12}$&$1.5012\times 10^{-11}$&7.71\\
\hline
40000&$7.9812\times 10^{-9}$&$0.9102\times 10^{-12}$&$7.1360\times 10^{-12}$&19.21\\
\hline
60000&$2.7512\times 10^{-10}$&$7.2361\times 10^{-12}$&$6.6641\times 10^{-12}$&32.42\\
\hline
80000&$0.4407\times 10^{-11}$&$5.2415\times 10^{-11}$&$7.1110\times 10^{-12}$&47.33\\
\hline \hline

\end{tabular}
\end{table}
\end{center}

\section{Conclusions}
In this paper, we  have studied  the problem of minimizing a
general quadratic function subject to an unit ball with an additional linear inequality constraint. We have restated
the necessary and sufficient condition for strong duality  by Ai and Zhang \cite{20} in terms of an optimal solution of a linearly constrained  bivariate concave maximization problem by deriving a new dual problem to eTRS based on a two-parameter linear eigenvalue problem. Moreover, we have proposed an efficient algorithm for solving large sparse eTRS  instances whenever strong duality is detected.
Our computational   experiments on several randomly generated test problems show that the proposed method is
always successful in finding  the global optimal solution of the underlying problem. Moreover, for small instances where we can compare it with the  SOCO/SDO reformulation show that the new approach is much faster than the reformulation.

%
%
%
\section*{Appendix}
\begin{proposition}
 Let $(t^{*},\lambda^{*})$ be an optimal solution of problem (\ref{eigp})  with $\lambda_{\text{min}}(D(t^{*},\lambda^{*}))=\lambda_{\text{min}}(A)$. Moreover, let $y(t^{*},\lambda^{*})=[y_0(t^{*},\lambda^{*}), z_0(t^{*},\lambda^{*})]^T$ be a normalized eigenvector for $\lambda_{\text{min}}(D(t^{*},\lambda^{*}))$ with $y_0(t^{*},\lambda^{*})\neq 0$. Then $||x^{*}||^2\leq \Delta $ where $x^{*}=\frac{z(t^{*},\lambda^{*})}{y_0(t^{*},\lambda^{*})}$.
 \end{proposition}
 \begin{proof}
 First notice  that, as discussed in proof of Item 1 of Theorem \ref{t4}, we have
 $$||x^{*}||^2=\frac{1-y_0(t^{*},\lambda^{*})^2}{y_0(t^{*},\lambda^{*})^2}.$$
On the other hand,  since $(t^{*},\lambda^{*})$ is an optimal solution of problem (\ref{eigp}), we have
$$t^{*}\in \text{argmax}_t\quad k(t,\lambda^{*}):= (\Delta+1)\lambda_{\text{min}}(D(t,\lambda^{*}))-t -\lambda^{*} c. $$
Moreover, we know, as shown in proof of Lemma \ref{le}, $t^{*}=t_0$ where $t_0$ is defined as before. The function  $k(t,\lambda^{*})$ is not differentiable at $t_0$ and the directional derivative from left, $k'(t_0^{-})=(\Delta +1)y_0(t^{*},\lambda^{*})^2-1$. The optimality of $t_0$ implies that $k'(t_0^{-})\geq 0$. This proves that $||x^{*}||^2\leq \Delta$.

\end{proof}
\begin{proposition}
Suppose that  $\lambda_{\text{min}}(A)$ has multiplicity, $i>2$ and
 let $\lambda_1^{*}:=-\lambda_{\text{min}}(A)\geq 0$ and $\lambda_2^{*}\geq 0$ be optimal Lagrange multipliers corresponding to the norm and the linear constraint of eTRS, respectively. Then $\lambda_1^{*}$ and $\lambda_2^{*}$ are also the optimal Lagrange multipliers corresponding to the norm and the linear constraint of eTRS with  $A$ replaced by $A+\alpha vv^T$ where $v$ is a normalized eigenvector for $\lambda_{\text{min}}(A)$ such that $v^Tb=0$ and $\alpha >0$.
\end{proposition}
\begin{proof}
Let $v$ be a normalized eigenvector for $\lambda_{\text{min}}(A)$ such that $v^Tb=0$. Moreover, let $A=PDP^T$ be the  eigenvalue decomposition of $A$ in which $P$ contains $v$. First,  since $\lambda_1^{*}=-\lambda_{\text{min}}(A)$ and $\lambda_2^{*}$ are the optimal Lagrange multipliers for eTRS, we have
$(a-\frac{\lambda_2^{*}}{2}b)\in \text{Range}(A+\lambda_1^{*} I).$
This implies that $v^T(a-\frac{\lambda_2^{*}}{2}b)=0$  which results in $v^Ta=0$. As strong duality holds for eTRS, we know that  $\lambda_1^{*}$ and $\lambda_2^{*}$ are the optimal solution of the Lagrangian dual of eTRS.  Hence, to prove the statement, it is enough to show that the Lagrangian dual of eTRS is equivalent to the one for eTRS with $A$ replaced by $A+\alpha vv^T$ where $\alpha$ is a positive constant. The Lagrangian dual of eTRS is given by
\begin{align}\label{d}
\max \quad & -(a-\frac{\lambda_2}{2}b)^T(A+\lambda_1 I )^{\dag}(a-\frac{\lambda_2}{2}b)-\lambda_1\Delta-\lambda_2 c\notag\\
& \quad (a-\frac{\lambda_2}{2}b)\in \text{Range}(A+\lambda_1 I),\\
& \quad \quad  A+\lambda_1 I\succeq0, \notag\\
&\quad \quad \lambda_1\geq, \lambda_2\geq 0.\notag
\end{align}
Now let $\lambda_1$ and $\lambda_2$ be feasible for problem (\ref{d}). Then it is easy to verify the following facts:
\begin{align}
&A+\lambda_1 I \succeq 0 \Leftrightarrow A+\alpha vv^T+\lambda_1 I \succeq 0,\\
& (a-\frac{\lambda_2}{2}b)\in \text{Range}(A+\lambda_1 I)\Leftrightarrow (a-\frac{\lambda_2}{2}b)\in \text{Range}(A+\alpha vv^T+\lambda_1 I),\\
&(A+\lambda_1 I )^{\dag}(a-\frac{\lambda_2}{2}b)=(A+\alpha vv^T+\lambda_1 I )^{\dag}(a-\frac{\lambda_2}{2}b).
\end{align}
Equation (17) uses the fact that $v^Ta=v^Tb=0$. It follows from (15),(16) and (17) that $\lambda_1^{*}$ and $\lambda_2^{*}$ are also the optimal Lagrange multipliers  for eTRS with $A+\alpha vv^T$.
\end{proof}

\section{Acknowledgments}The authors would like to thank Prof. Henry Wolkowicz for his useful suggestions. The first author would like to thank University of Guilan for the financial support during his Sabbatical at the University of Waterloo, Canada.


\begin{thebibliography}{99}
\bibitem{20} W. Ai  and S. Zhang, Strong duality for the CDT subproblem: a necessary and sufficient condition, SIAM Journal on Optimization 19.4 (2009): 1735-1756.
\bibitem{14} A. Beck and Y. C. Eldar, Strong duality in nonconvex quadratic optimization with two quadratic constraints, SIAM Journal on Optimization,  17(3), 844-860, 2006.
\bibitem{4} A. Ben-Tal, L. E.  Ghaoui and A. Nemirovski,  Robust Optimization,  Princeton Series
in Applied Mathematics, 2009.
 \bibitem{5} D. Bertsimas,  D. Pachamanova and M.  Sim,  Robust linear optimization under general norms,  Oper.  Res. Lett, 32(6), 510-516, 2004.
\bibitem{6} D. Bertsimas,  D.  Brown  and C.  Caramanis, Theory and applications of robust optimization, SIAM review, 53(3),  464-501, 2011.
\bibitem{3} S. Burer and   K. M. Anstreicher, Second-order-cone constraints for extended trust-region subproblems,  SIAM J. Optim,  23(1),  432-451, 2013.
\bibitem{18} S. Burer and B. Yang, The trust region subproblem with non-intersecting linear constraints,  Mathematical Programming 149, 253-264, 2015.
\bibitem{8}T.F. Coleman and  A. Liao,  An efficient trust region method for  unconstrained discrete-time optimal control problems,  Computational Optimization and Applications, 4, 47-66, 1995.
\bibitem{1} A. R. Conn, N. I.  Gould and  P. L. Toint,  Trust region methods,  SIAM, Philadelphia, PA, 2000.
\bibitem{7} M. Fukushima  and Y. Yamamoto, A second-order algorithm for continuous-time nonlinear optimal control problems, IEEE Transactions on  Automatic Control, 31(7),  673-676, 1986.
\bibitem{9}C. Fortin and H.  Wolkowicz, The trust region subproblem and semidefinite programming, Optimization Methods and Software,  19(1), 41-67, 2004.
\bibitem{10} N. I. Gould, S. Lucidi, M.  Roma and P. L.  Toint, Solving the trust-region subproblem using the Lanczos method, SIAM Journal on Optimization, 9(2), 504-525, 1999.
\bibitem{11} N. I. Gould, D. P.  Robinson and H. S.  Thorne, On solving trust-region and other regularised subproblems in optimization, Mathematical Programming Computation, 2(1), 21-57, 2010.
 \bibitem{17} Y. Hsia and R. L. Sheu, Trust region subproblem with a fixed number of additional linear inequality constraints has polynomial complexity,  arXiv preprint arXiv, 1312.1398, 2013.
\bibitem{16} V. Jeyakumar and G. Y. Li, Trust-region problems with linear inequality constraints: exact SDP relaxation, global optimality and robust optimization, Mathematical Programming 147,171-206, 2014.

\bibitem{19}M. L. Overton,  Large-scale optimization of eigenvalues, SIAM Journal on Optimization,  2(1),  88-120, 1992.
\bibitem{30}
U. Niesen, S. Devavrat and W. Gregory,  Adaptive alternating minimization algorithms,   IEEE Transactions on Information Theory,  55(3), 1423-1429, 2009.

\bibitem{2} P. Pardalos and H. Romeijn, Handbook of Global Optimization, vol. 2, Kluwer Academic Publishers, Dordrecht, the Netherlands, 2002.
\bibitem{12} F. Rendl and H. Wolkowicz,  A semidefinite framework for trust region subproblems with applications to large scale minimization,  Mathematical Programming,  77(1), 273-299, 1997.
\bibitem{sf}M. Salahi, S. Fallahi, Trust region subproblem with an additional linear
inequality constraint, Optimization Letters, DOI 10.1007/s11590-015-0957-5, 2015.
\bibitem{15}J. F. Sturm and S. Zhang, On cones of nonnegative quadratic functions, Mathematics of Operations Research, 28(2),  246-267,  2003.
\bibitem{13} Y. Ye,  and S.  Zhang,  New results on quadratic minimization. SIAM Journal on Optimization,  14(1),  245-267, 2003.

\end{thebibliography}
\end{document}